\documentclass[11pt,a4paper]{amsart}
\usepackage[english]{babel} 
\usepackage{hyperref} 
\usepackage{amsfonts,amscd,amsmath,amssymb,amsthm} 
\usepackage{xparse} 
\usepackage{dsfont} 
\usepackage[utf8]{inputenc} 

\numberwithin{equation}{section}

\newtheorem{theorem}{Theorem}[section]{\bf}{\it}
\newtheorem{proposition}[theorem]{Proposition}{\bf}{\it}
\newtheorem{corollary}[theorem]{Corollary}{\bf}{\it}
\newtheorem{lemma}[theorem]{Lemma}{\bf}{\it}

\theoremstyle{definition}
\newtheorem{definition}[theorem]{Definition}
\newtheorem{remark}[theorem]{Remark}
\newtheorem{example}[theorem]{Example}

\theoremstyle{plain}
\newcounter{countertheoremintro}
\newtheorem{introtheorem}[countertheoremintro]{Theorem}

\newcommand{\idempotents}{E}
\newcommand{\spectrum}{\hat{\idempotents}_0}
\newcommand{\univgrpd}{G_u(S)}
\newcommand{\hrel}{\mathrel{\mathcal{H}}}
\newcommand{\lrel}{\mathrel{\mathcal{L}}}
\newcommand{\rrel}{\mathrel{\mathcal{R}}}
\newcommand{\drel}{\mathrel{\mathcal{D}}}

\title{A note on the quasi-diagonality of inverse semigroup reduced C*-algebras}

\author[Diego Mart\'{i}nez]{Diego Mart\'{i}nez $^{1}$}
\address{Mathematisches Institut, WWU M\"{u}nster, Einsteinstr. 62, 48149 M\"{u}nster, Germany.}
\email{diego.martinez@uni-muenster.de}

\date{\today}
\thanks{{$^{1}$} Partly funded by research projects MTM2017-84098-P, Severo Ochoa SEV-2015-0554 and BES-2016-077968 of the Spanish Ministry of Economy and Competition (MINECO), Spain. Partly funded by the Deutsche Forschungsgemeinschaft (DFG, German Research Foundation) under Germany’s Excellence Strategy – EXC 2044 – 390685587, Mathematics Münster – Dynamics – Geometry – Structure; the Deutsche Forschungsgemeinschaft (DFG, German Research Foundation) – Project-ID 427320536 – SFB 1442, and ERC Advanced Grant 834267 - AMAREC}
\keywords{Inverse semigroup, quasi-diagonality, amenability}
\subjclass[2020]{20M18, 46L05, 47A66, 46L89}

\begin{document}

  \begin{abstract}
    In this note we start the study of whether the reduced C*-algebra of an inverse semigroup is quasi-diagonal, making explicit use of the inner structure of this class of semigroups in order to produce quasi-diagonal approximations. Given a discrete inverse semigroup, we detail the relationship between its isolated subgroups and the quasi-diagonality of its reduced C*-algebra, and prove that such subgroups must be amenable. Moreover, we give a direct characterization of the quasi-diagonality of an inverse semigroup whose universal groupoid is minimal. Lastly, we also study the relevance of Green's $\mathcal{D}$-relation when considering quasi-diagonality questions, and give a sufficient condition for the quasi-diagonality of a general inverse semigroup.
  \end{abstract}
  \maketitle

  \begin{center}
    \textit{Dedicated to the memory of Paqui, Lali and Pilar}
  \end{center}
  \vspace{1mm}

  \section{Introduction}\label{sec:intro}
  Quasi-diagonality was introduced by Halmos~\cite{H69} with the goal of studying the invariant subspaces of a fixed bounded linear operator on an infinite-dimensional Hilbert space. He intended to study the invariant subspaces of a large class of operators, including the compact operators and the normal ones, which were known to have invariant subspaces. Roughly speaking, the very name \textit{quasi-diagonal} refers to the fact that these operators are almost diagonal in some basis, meaning that there is a basis of the Hilbert space on which the operator almost acts. In mathematical terms, we say an operator is \textit{quasi-diagonal} if there is an increasing sequence of finite-dimensional projections that converge strongly to the identity and asymptotically commute with the original operator. Among other things, it was quickly realized that if an operator is quasi-diagonal then so is the concrete C*-algebra it generates, and hence quasi-diagonality became a notion of great relevance within the C*-community. In particular, Brown defined and studied the so-called \textit{quasi-diagonal traces} in~\cite{B06}, and boldly predicted that these traces (and quasi-diagonality in general) were to have an essential role within the classification program of nuclear simple C*-algebras. This insight was indeed well founded, as the stably finite half of the classification program now hinges upon these quasi-diagonal traces (see~\cite{TWW17} and references therein) and how they can be approximated.

  Of particular interest to this paper is the theorem proved by Rosenberg in an appendix to~\cite{H87}, where he showed that, among discrete groups, only those that are amenable can have quasi-diagonal reduced C*-algebras. The converse became known as \textit{Rosenberg's conjecture}, and was only proved in the spectacular paper~\cite{TWW17} (see also~\cite{G17} for a generalization and~\cite{S17} for an alternative proof). In there, it is proved that nuclear separable C*-algebras that satisfy the UCT and have a faithful trace must be quasi-diagonal. For the case of reduced group C*-algebras note that nuclearity and the UCT follow from the amenability of the group (see, respectively,~\cite[Theorem~2.6.8]{BO08} and~\cite[Proposition~10.7]{Tu99}). Recall, moreover, that any reduced group C*-algebra has a canonical faithful trace (see, for instance,~\cite{BO08,TWW17}). The goal of this note is hence to begin the study of quasi-diagonality in the more general setting of reduced C*-algebras coming from discrete inverse semigroups. In particular, we aim to use the inner structure of the semigroup in order to provide both necessary and sufficient conditions for the quasi-diagonality of $C_r^*(S)$.

  A few key aspects of the point of view of the text are the following. First, note that even though inverse semigroups may behave very differently from groups, they do still retain a certain \textit{group-like} structure (see, for instance,~\cite{G51} or Section~\ref{sec:pre} below), which can be exploited in order to produce quasi-diagonal approximations, among other types of approximations (see, e.g.,~\cite{ALM19,LM21}). Moreover, reduced C*-algebras of inverse semigroups are much wilder than those of groups (see, e.g., the seminal~\cite{E08}). This freedom actually entails we may no longer automatically have a faithful trace, so a new approach is desirable if one wishes to study quasi-diagonality of these algebras. For instance, note that $C_r^*(S)$ might be quasi-diagonal but have no faithful trace, as Example~\ref{ex:qdnotr} below shows. It is our opinion that inverse semigroups might be a suitable framework from whence to study these quasi-diagonality questions.

  We now indicate the main contributions of the paper. The first of these says that if $C^*_r(S)$ is quasi-diagonal then certain subgroups of $S$ must be amenable, which can be seen as a generalization of Rosenberg's contribution of~\cite{H87}.
  \begin{introtheorem}[cf. Theorem~\ref{thm:subgrp}] \label{introthm:subgrp}
    Let $S$ be a discrete inverse semigroup, and let $H \subset S$ be a subgroup whose unit is isolated in the spectrum of $S$. If $C_r^*(S)$ is quasi-diagonal, then $H$ is amenable.
  \end{introtheorem}
  Furthermore, in Example~\ref{ex:qd-nonfl} we construct an inverse semigroup with quasi-diagonal reduced C*-algebra and with a non-amenable subgroup. As mentioned before, Example~\ref{ex:qdnotr} below shows that $C_r^*(S)$ may be quasi-diagonal but have no faithful trace. This problem has also appeared in the literature, for instance in~\cite[Theorem~6.5]{RS19}, but has been resolved by considering only minimal groupoids, as in those cases every trace is automatically faithful. The same kind of result can be shown from our point of view.
  \begin{introtheorem}[cf. Theorems~\ref{thm:grpdmin} and~\ref{thm:minimalqd}] \label{introthm:minimal}
    Let $S$ be an inverse semigroup whose universal groupoid is minimal (as a groupoid). Then $C_r^*(S)$ is quasi-diagonal if, and only if, every subgroup of $S$ is amenable and $C_r^*(S)$ is finite.
  \end{introtheorem}
  The last main contribution of the paper gives an explicit construction (modulo group C*-algebras) of quasi-diagonal approximations in a large class of inverse semigroups (for the definition of the $\mathcal{D}$-relation see Section~\ref{sec:pre}).
  \begin{introtheorem}[cf. Theorem~\ref{thm:dclasses}] \label{introthm:dclasses}
    Let $S$ be a discrete inverse semigroup. Suppose every subgroup of $S$ is amenable and no $\mathcal{D}$-class contains infinitely many idempotents. Then $C_r^*(S)$ is quasi-diagonal.
  \end{introtheorem}

  The paper is structured as follows. Section~\ref{sec:pre} recalls the necessary background of the paper, including definitions and some theorems used throughout the text. Section~\ref{sec:subgrp} then proves Theorem~\ref{introthm:subgrp}, while Section~\ref{sec:minimal} is devoted to Theorem~\ref{introthm:minimal}. Note we regard both Theorems~\ref{introthm:subgrp} and~\ref{introthm:minimal} as necessary conditions for the quasi-diagonality of $C_r^*(S)$. Lastly, Section~\ref{sec:dclasses} proves Theorem~\ref{introthm:dclasses}.

  \textbf{Conventions:} throughout the text $S$ will denote a discrete inverse semigroup with a zero element, denoted by $0$. The set of idempotents, or projections, of $S$ will be denoted by $\idempotents$, while $\univgrpd$ will be the universal groupoid associated to $S$. Moreover, $\spectrum$ shall denote the unit space of $\univgrpd$, that is, the space of filters of $\idempotents$ equipped with pointwise convergence. Lastly, $\mathcal{B}(\ell^2(S))$ will denote the algebra of bounded linear operators on the complex separable Hilbert space $\ell^2(S)$, which is formed by $2$-summable complex sequences indexed by $S$. The usual orthonormal basis of $\ell^2(S)$ shall be denoted by $\left\{\delta_x\right\}_{x \in S}$.

  \textbf{Acknowledgements:} part of this research was conducted while visiting Texas A\&M invited by David Kerr, to whom the author is deeply grateful. Moreover, the author is also thankful to the anonymous referees for their helpful comments on a previous version of the manuscript. 

  \section{Preliminaries} \label{sec:pre}
  Inverse semigroups were introduced independently by Wagner (see~\cite{W52,W53}) and Preston (see~\cite{P54a,P54b,P54c}) as a generalization of groups. While elements in a group are usually thought of as symmetries of a space, or as bijections of a set, elements of an inverse semigroup can be thought of as partial symmetries of a space, or as partial bijections of a set. For the following, recall that a \textit{semigroup} is a set $S$ equipped with a binary and associative operation.
  \begin{definition}
    A semigroup $S$ is \textit{inverse} if for all $s \in S$ there is a unique $s^* \in S$ such that $ss^*s = s$ and $s^*ss^* = s^*$.
  \end{definition}
  We say an element $e \in S$ is an \textit{idempotent}, or \textit{projection}, if $e^2 = e$. It is well known that the set $\idempotents$ of idempotents of $S$ forms a commutative subsemigroup (see, for instance,~\cite[Theorem~3]{L98}). Moreover, note that we may equip $S$ with a natural partial order, where we say $s \geq t$ if there is some $e \in \idempotents$ such that $se = t$. Observe, in addition, this is equivalent to the existence of an $f \in \idempotents$ such that $fs = t$. Furthermore, the partial order $\leq$, when restricted to idempotents, can be reformulated to $e \leq f$ if $ef = e$, which we usually describe as $e$ being contained in $f$. The \textit{unit} of $S$, which, if it exists, we denote by $1$, is the unique element such that $1 \cdot s = s \cdot 1 = s$ for every $s \in S$. Likewise, the \textit{zero} of $S$, which we denote by $0$, is the unique element such that $0 \cdot s = s \cdot 0 = 0$ for every $s \in S$. Observe that we do not require $S$ to be unital, but we do require it to contain a zero element throughout the text.

  The so-called \textit{Green's relations} (see, e.g.,~\cite{G51,H76,H95,L98}) are useful tools in the study of semigroup theory, and have a particularly simple description for inverse semigroups. In this way, given $s, t \in S$, we say that $s \lrel t$ if $s^*s = t^*t$, and $s \rrel t$ if $ss^* = tt^*$. Likewise, Green's relation $\mathcal{H}$ is defined as $\mathcal{L} \cap \mathcal{R}$, while $\mathcal{D}$ is the join of $\mathcal{L}$ and $\mathcal{R}$. For instance, it is routine to show that two idempotents $e, f \in E(S)$ are $\mathcal{D}$-related precisely when $e = s^*s$ and $f = ss^*$ for some $s \in S$. All $\mathcal{L}, \mathcal{R}, \mathcal{H}$ and $\mathcal{D}$ are easily seen to be equivalence relations, and, thus, their equivalence classes will be called \textit{$\mathcal{L}, \mathcal{R}, \mathcal{H}$} or \textit{$\mathcal{D}$-classes}, respectively.

  We next point out two important classes of examples for the context of this paper.
  \begin{example} \label{ex:bisimple}
    Let $\idempotents^+$ be a non-empty set, and let $H$ be a discrete group. Then the set $S := \idempotents^+ \times H \times \idempotents^+ \sqcup \{0\}$ can be equipped with the operation given by
    \begin{equation}
      (f_2, h_2, e_2)(f_1, h_1, e_1) := 
            \left\{
              \begin{array}{rl}
                (f_2, h_2 h_1, e_1) & \text{if} \;\; e_2 = f_1, \\
                0 & \text{otherwise}
              \end{array}
            \right. \nonumber
        \end{equation}
    for every $(f_2, h_2, e_2), (f_1, h_1, e_1) \in \idempotents^+ \times H \times \idempotents^+$. It is not hard to see that then $S$ is an inverse semigroup. Note that, even though this example might seem too simple and ad-hoc, it appears naturally in Section~\ref{sec:minimal}, in particular with regards to Theorems~\ref{thm:grpdmin} and~\ref{thm:minimalqd}.
  \end{example}
  \begin{example} \label{ex:groupbundle}
    Let $\idempotents$ be any semi-lattice (i.e., a partially ordered set with intersection $\wedge$), and let $G_e$ be a discrete group for every $e \in E$. Moreover, suppose there are morphisms $\pi_{e,f} \colon G_e \rightarrow G_f$ for every ordered pair $e \geq f$. Lastly, assume that $\pi_{e_1,e_1}$ is the identity of $G_{e_1}$ and $\pi_{e_1,e_3} = \pi_{e_2, e_3} \circ \pi_{e_1, e_2}$ for every triple $e_1 \geq e_2 \geq e_3$. Then the set $S := \sqcup_{e \in E} G_e$ can be equipped with a natural inverse semigroup structure given by $g_e \cdot h_f := \pi_{e, e \wedge f}(g_e) \pi_{f, e \wedge f}(h_f) \in G_{e \wedge f}$. In particular, these semigroups are called \textit{Clifford semigroups} (see~\cite{L98}). Actually, these are characterized by the fact that every idempotent is central in $S$ or, equivalently, the fact that $\lrel = \rrel = \hrel = \drel$.
  \end{example}

  As we have mentioned, inverse semigroups generalize groups, in the sense that if a group can be seen as a set of bijections of a space, then an inverse semigroup is a set of partial bijections of a space. This point of view can be properly defined by considering the so-called \textit{Wagner-Preston representation} (see~\cite{P54a,W52}). Such representation can be implemented in $\ell^2(S)$ via the \textit{left regular representation} $v$ of $S$, that is:
  \begin{equation}
    v \colon S \rightarrow \mathcal{B}\left(\ell^2\left(S\right)\right), \quad
    v_s \delta_x :=
      \left\{
        \begin{array}{lr}
          \delta_{sx} & \text{if} \; s^*s x = x, \\
          0 & \text{otherwise.}
        \end{array}
      \right. \nonumber
  \end{equation}
  Likewise, the \textit{right regular representation} $w$ of $S$ is defined by multiplying on the right:
  \begin{equation}
    w \colon S \rightarrow \mathcal{B}\left(\ell^2\left(S\right)\right), \quad
    w_s \delta_x :=
      \left\{
        \begin{array}{rl}
          \delta_{xs^*} & \text{if} \; xs^*s = x, \\
          0 & \text{otherwise}.
        \end{array}
      \right. \nonumber
  \end{equation}
  The C*-algebra $C_r^*(S)$ generated by the left regular representation in $\mathcal{B}(\ell^2(S))$ is the \textit{reduced C*-algebra of $S$}, i.e., $C_r^*(S) := C^*(\{v_s\}_{s \in S})$. It is the main object of study of the paper.

  Groupoids have been closely related to inverse semigroups ever since the pioneer work~\cite[Chapter~4]{P12} (see also~\cite{AR00,E08} and references therein). In~\cite{P12}, starting from an inverse semigroup $S$, Paterson constructed a universal groupoid that inherits all its representation theory. This groupoid, which henceforth shall be denoted $\univgrpd$, is usually called either the \textit{universal groupoid} of $S$ or \textit{Paterson's groupoid} associated to $S$. We now outline its construction (see~\cite[Chapter~4]{P12} for a detailed discussion). Given an inverse semigroup $S$ with zero, let $\idempotents$ be its set of idempotents equipped with the natural partial order $\geq$. Recall that a \textit{filter} is a non-empty $\xi \subset \idempotents$ that is non-trivial, i.e., $0 \not\in S$; upwards closed, i.e., if $e \in \xi$ and $e \leq f$ then $f \in \xi$; and closed under multiplication, i.e., if $e, f \in \xi$ then $ef \in \xi$. The \textit{spectrum} of $S$, which we denote by $\spectrum$, is the set of filters equipped with pointwise convergence, or, equivalently, the relative topology of $\spectrum$ as a subset of $\{0, 1\}^\idempotents$. It is routine to show that the space $\spectrum$ is a locally compact, totally disconnected, second countable Hausdorff space. Letting $D_e^\theta \subset \spectrum$ be the set of filters that contain $e$, the semigroup $S$ acts on $\spectrum$ in the following way:
  \begin{equation}
    \theta_s \colon D_{s^*s}^\theta \rightarrow D_{ss^*}^\theta, \;\; \theta_s\left(\xi\right) := \left\{e \in \idempotents \mid e \geq sfs^* \;\, \text{for some} \,\; f \in \xi\right\}. \nonumber
  \end{equation}
  The \textit{universal groupoid of $S$} is the groupoid of germs of this action, that is
  \begin{equation}
    \univgrpd := \left\{\left[s, \xi\right] \;\; \text{where} \;\; s \in S \;\, \text{and} \,\; \xi \in D_{s^*s}^\theta\right\}, \nonumber
  \end{equation}
  where two germs $[s, \xi], [t, \zeta] \in \univgrpd$ are equal if, and only if, $\xi = \zeta$ and $se = te$ for some $e \in \idempotents$ such that $e \in \xi$. Lastly, we shall always equip $\univgrpd$ with the topology generated by sets of the form $\{[s, \xi] \mid \xi \in U\}$, where $s \in S$ and $U \subset D_{s^*s}^\theta$ is open.

  Of particular interest for the purposes of this text is the natural map $S \rightarrow \univgrpd$ sending $s \mapsto [s, (s^*s)^\uparrow]$, where $e^\uparrow := \{f \in \idempotents \mid f \geq e\} \in \spectrum$. Henceforth, we may tacitly consider $S$ or $\idempotents$ as contained in $\univgrpd$, particularly in Lemma~\ref{lemma:isolated} and the subsequent Definition~\ref{def:isolatedsbgp}.

  We end the section recalling the necessary background for C*-algebras, mostly about \textit{quasi-diagonality} and related topics. The following goes back to Halmos~\cite{H69}.
  \begin{definition} \label{def-qd}
    We say a separable C*-algebra $A$ is \textit{quasi-diagonal} if it has a faithful representation $\pi \colon A \rightarrow \mathcal{B}(\mathcal{H})$ such that there is a sequence $\{p_n\}_{n \in \mathbb{N}} \subset \mathcal{B}(\mathcal{H})$ of finite rank orthogonal projections that asymptotically centralize $A$ and converge strongly to the identity.
  \end{definition}
  It is well known that abelian C*-algebras are quasi-diagonal, as is the algebra of compact operators of a Hilbert space. More interestingly, the following was proved in the appendix to~\cite{H87} and~\cite[Corollary~C]{TWW17}.
  \begin{theorem} \label{thm:tww:groups}
    A discrete group $G$ is amenable if, and only if, $C_r^*(G)$ is quasi-diagonal.
  \end{theorem}
  As Brown and Ozawa mention in~\cite[Chapter~7]{BO08}, there are two main obstructions to quasi-diagonality. The most relevant one for the purposes of this paper is that quasi-diagonal C*-algebras are stably finite (see~\cite[Proposition~7.1.15]{BO08}).
  \begin{proposition} \label{prop:qd-sf}
    Every quasi-diagonal C*-algebra is stably finite.
  \end{proposition}
  Lastly, the following proposition includes some sufficient conditions that guarantee stable finiteness. The proof of the first assessment can be found in~\cite[Proposition~V.2.1.8]{Bl13}, while the second is known to experts.
  \begin{proposition} \label{prop:sf}
    Let $A$ be a separable C*-algebra.
    \begin{enumerate}
      \item If $A$ is an inductive limit of unital stably finite C*-algebras then $A$ itself is stably finite.
      \item Let $p \in A$ be a projection. If $A$ is finite (resp. stably finite), then so is $pAp$.
    \end{enumerate}
  \end{proposition}

  \section{Subgroups of quasi-diagonal inverse semigroups} \label{sec:subgrp}
  This section gives the first necessary condition for $C_r^*(S)$ to be quasi-diagonal. Particularly, in Theorem~\ref{thm:subgrp} we prove that certain subgroups of $S$ must be amenable whenever $C_r^*(S)$ is quasi-diagonal. For the following, observe that if $H \subset S$ is a subgroup then $H \cap \idempotents$ has exactly $1$ element, namely the unit of $H$.
  \begin{lemma} \label{lemma:isolated}
    Let $H \subset S$ be a subgroup, and let $e \in H \cap \idempotents$. The following are equivalent:
    \begin{enumerate}
      \item \label{lemma:isolated:isolated} The filter $e^\uparrow$ is an isolated point of $\spectrum$.
      \item \label{lemma:isolated:idempotents} There are idempotents $f_1, \dots, f_k \in E$ such that $f_i \leq e$ for every $i = 1, \dots, k$ and for every $f \in E$ such that $f \leq e$ there is some $i_0 \in \{1, \dots, k\}$ such that $f \leq f_{i_0}$.
    \end{enumerate}
  \end{lemma}
  \begin{proof}
    For convenience, we prove that the negation of condition~(\ref{lemma:isolated:isolated}) is equivalent to the negation of~(\ref{lemma:isolated:idempotents}). Note that, as $\idempotents \subset \spectrum$ is dense, a point $e^\uparrow \in \spectrum$ is not isolated if, and only if, there are filters $e_n^\uparrow \in \spectrum$ converging pointwise to $e^\uparrow$. Therefore $e^\uparrow$ is not isolated if, and only if, there are idempotents $e_n \in E$ such that for any idempotent $f \in E$ one has that $f \geq e$ if, and only if, $f \geq e_n$ for all but finitely many $n \in \mathbb{N}$. In particular, note that we may suppose that $e_n \leq e$ for all $n \in \mathbb{N}$. Thus for any finitely many $\{e_1, \dots, e_k\}$ there is some idempotent $e_{k + 1}$ with $e_{k+1} \leq e$ but $e_{k+1} \not\leq e_i$ for all $i = 1, \dots, k$, which proves the claim.
  \end{proof}
  \begin{definition} \label{def:isolatedsbgp}
    We say a subgroup $H \subset S$ is \textit{isolated} if it satisfies any condition of Lemma~\ref{lemma:isolated}.
  \end{definition}

  The following proposition shows the relevance of the isolated subgroups in our context.
  \begin{proposition} \label{prop:isolated}
    Let $S$ be an inverse semigroup, and let $H \subset S$ be a subgroup. If $H$ is isolated then the inclusion $H \subset S$ extends to a canonical embedding $C_r^*(H) \subset C_r^*(S)$.
  \end{proposition}
  \begin{proof}
    By~\cite[Proposition~2.5.9]{BO08} the claim holds when $S$ is a group, and since every subgroup is contained in an $\mathcal{H}$-class we may suppose that $H$ is actually a maximal subgroup of $S$, i.e., an $\mathcal{H}$-class. Let $e \in H \cap E$ be the (necessarily unique) idempotent in $H$, and let $f_1, \dots, f_k \in E$ be those given by Lemma~\ref{lemma:isolated}. Starting with the projection $p_1 = v_{f_1}$, iteratively produce $p_i = p_{i-1} + v_{f_i} - p_{i-1}v_{f_i}$. It is tedious but routine to show that $p = p_k$ commutes with $v_h$ for every $h \in H$ (note that $h$ may not commute with $f_i$ for any $i$). Moreover, observe that $v_e - p$ is also a projection, this time with image $\ell^2(H) \subset \ell^2(S)$. Consider then the representation
    \begin{equation}
      \pi \colon \mathbb{C} H \rightarrow C_r^*(S), \;\; \sum_{h \in H} a_h h \mapsto \left(\sum_{h \in H} a_h v_h\right) (v_e - p), \nonumber
    \end{equation}
    and observe it is sufficient to prove that $\pi$ is continuous with respect to the reduced norm in $\mathbb{C} H$, i.e., $||\pi(a)|| \leq ||a||_r$ for every $a \in \mathbb{C} H$, where $||\cdot||_r$ denotes the norm in $C_r^*(H)$. For this, if $w \in \ell^2(S)$ is such that $||\pi(a) w|| \geq ||\pi(a)|| - \varepsilon$ then $||pw|| \leq \varepsilon$, and therefore we may find a vector $\tilde{w} \in \ell^2(H) = (v_e - p)(\ell^2(S))$ such that $||a \tilde{w}|| \geq ||\pi(a)|| - 2 \, \varepsilon$.
  \end{proof}
  \begin{remark}
    Observe that the key point of Proposition~\ref{prop:isolated} is that there is a projection $q \in C_r^*(S)$, namely $q = v_e - p$, that commutes with $\{v_h\}_{h \in H}$, and whose image is precisely $\ell^2(H)$. This actually holds in a slightly more general setting than that of Proposition~\ref{prop:isolated}. Indeed, for example this argument also holds when there is an idempotent $f < e$ that centralizes $H$ and with no other idempotent in between, i.e., there is no $f'$ such that $f < f' < e$.
  \end{remark}

  With Proposition~\ref{prop:isolated} at hand the following is immediate.
  \begin{theorem} \label{thm:subgrp}
    Let $S$ be a discrete inverse semigroup, and let $H \subset S$ be an isolated subgroup. If $C_r^*(S)$ is quasi-diagonal, then $H$ is amenable.
  \end{theorem}
  \begin{proof}
    By Proposition~\ref{prop:isolated} we have that $C_r^*(H) \subset C_r^*(S)$, and quasi-diagonality clearly passes to sub-C*-algebras. Therefore $H$ is amenable by Rosenberg's contribution of~\cite{H87}.
  \end{proof}
  \begin{remark}
    Example~\ref{ex:qd-nonfl} below shows that the group $H$ needs to be isolated in $S$ for the conclusion to hold. Indeed, it provides an inverse semigroup $S$ with a non-amenable subgroup and quasi-diagonal reduced C*-algebra.
  \end{remark}

  \section{Quasi-diagonal semigroups with minimal universal groupoid} \label{sec:minimal}
  We now study those inverse semigroups $S$ whose universal groupoid $\univgrpd$ is minimal (in the groupoid sense). This minimality condition has, for instance, appeared in~\cite[Theorem~6.5]{RS19}, where it is used to guarantee that a certain constructed trace is faithful. In this section, however, we prove that this approach is not suitable when tackling quasi-diagonality of inverse semigroups, since this method imposes a very restrictive set of conditions on the inverse semigroup $S$ (see Theorem~\ref{thm:grpdmin}). We first need the following lemma, whose proof we outline only for convenience of the reader and future reference.
  \begin{lemma} \label{lemma:hclass}
    Let $S$ be an inverse semigroup, and let $D \subset S$ be a $\mathcal{D}$-class. Any two $\mathcal{H}$-classes in $D$ are in bijective correspondence. Moreover, in case both $\mathcal{H}$-classes are subgroups then the bijection may be taken to be a group isomorphism.
  \end{lemma}
  \begin{proof}
    Fix an idempotent $e_0 \in D$ and, given any idempotent $f \in D$, let $r_f \in D$ be such that $r_f^*r_f = e_0$ and $r_f r_f^* = f$. We may further assume that $r_{e_0} = e_0$. Then, letting $H_{e_0}$ be the $\mathcal{H}$-class of $e_0$, it is routine to show that the map
    \begin{equation}
      \phi \colon D \rightarrow H_{e_0}, \;\; \text{where} \;\; s \mapsto r^*_{ss^*} s r_{s^*s} \nonumber
    \end{equation}
    restricts to the desired bijections in every $\mathcal{H}$-class $H \subset D$. Moreover, in case $H$ contains an idempotent (or, equivalently, is a maximal subgroup of $S$), $\phi$ actually restricts to a group isomorphism.
  \end{proof}
  The following proposition is worth mentioning independently of the main result. Recall that being finite is a necessary condition for quasi-diagonality (see Proposition~\ref{prop:qd-sf} above).
  \begin{proposition} \label{prop:finite}
    Let $S$ be an inverse semigroup such that $C_r^*(S)$ is finite. Then the partial order of $S$ restricts to equality in the $\mathcal{D}$-classes, that is, if $s \drel t$ and $s \geq t$ then $s = t$ for every $s, t \in S$.
  \end{proposition}
  \begin{proof}
    The proof boils down to the fact that finiteness is preserved under taking corners (see Proposition~\ref{prop:sf}). Indeed, suppose there are $s, t \in S$ such that $s \drel t$ and $s > t$. It is not hard to see that then $s^*s \drel t^*t$ and $s^*s > t^*t$. Letting $x \in S$ be such that $s^*s = x^*x$ and $t^*t = xx^*$ yields that the corner $v_{s^*s} C_r^*(S) v_{s^*s}$ is an infinite C*-algebra, and hence (by Proposition~\ref{prop:sf}) so is $C_r^*(S)$. This contradicts the finiteness hypothesis.
  \end{proof}

  Recall that we say an inverse semigroup $S$ is \textit{bisimple} if it has exactly one $\mathcal{D}$-class (see, for instance,~\cite[p.~85]{L98}). Likewise, $S$ is \textit{0-bisimple} if it is either bisimple or has a zero element and exactly two $\mathcal{D}$-classes. Note that in the latter case one of the $\mathcal{D}$-classes is necessarily formed by the zero element alone. Lastly, we say a groupoid $G$ is \textit{minimal} if the range of $G_x$ is dense in $G^{(0)}$ for every unit $x \in G^{(0)}$, where $G_x$ denotes the set of elements whose source is $x$.
  \begin{theorem} \label{thm:grpdmin}
    Let $S$ be a discrete inverse semigroup. The following assertions are equivalent:
    \begin{enumerate}
      \item \label{thm:grpd:minimal} The universal groupoid $\univgrpd$ is minimal and $C_r^*(S)$ is stably finite.
      \item \label{thm:grpd:finite} The universal groupoid $\univgrpd$ is minimal and $C_r^*(S)$ is finite.
      \item \label{thm:grpd:bisimple} $S$ is $0$-bisimple and the partial order in $S$ restricts to equality in the $\mathcal{D}$-classes of $S$.
      \item \label{thm:grpd:struc} Either $S$ is a group or it is isomorphic to $\idempotents^+ \times H \times \idempotents^+ \sqcup \{0\}$, where $H$ is the unique (up to isomorphism) non-trivial maximal subgroup, $\idempotents^+$ is the set of non-zero idempotents and
        \begin{equation}
          (f_2, h_2, e_2)(f_1, h_1, e_1) := 
            \left\{
              \begin{array}{rl}
                (f_2, h_2 h_1, e_1) & \text{if} \;\; e_2 = f_1, \\
                0 & \text{otherwise}
              \end{array}
            \right. \nonumber
        \end{equation}
        for every $(f_2, h_2, e_2), (f_1, h_1, e_1) \in \idempotents^+ \times H \times \idempotents^+$.
    \end{enumerate}
  \end{theorem}
  \begin{proof}
    The implication (\ref{thm:grpd:minimal}) $\Rightarrow$ (\ref{thm:grpd:finite}) is obvious. For (\ref{thm:grpd:finite}) $\Rightarrow$ (\ref{thm:grpd:bisimple}), by Proposition~\ref{prop:finite} we only have to prove that $S$ is $0$-bisimple. Thus, it suffices to show that any two non-zero idempotents $e, f \in \idempotents \setminus \{0\}$ are $\mathcal{D}$-related. As $\univgrpd$ is minimal there are elements $\{s_n\}_{n \in \mathbb{N}}, \{t_k\}_{k \in \mathbb{N}} \subset S$ such that
    \begin{align}
      e^\uparrow \in D_{s_n^*s_n}^\theta \; (\,\text{i.e.} \; s_n^*s_n \geq e) \;\; \text{and} \;\; r\left(\left[s_n, e^\uparrow\right]\right) = \left(s_nes_n^*\right)^\uparrow \xrightarrow{n \rightarrow \infty} f^\uparrow, \nonumber \\
      f^\uparrow \in D_{t_k^*t_k}^\theta \; (\,\text{i.e.} \; t_k^*t_k \geq f) \;\; \text{and} \;\; r\left(\left[t_k, f^\uparrow\right]\right) = \left(t_kft_k^*\right)^\uparrow \xrightarrow{k \rightarrow \infty} e^\uparrow. \nonumber
    \end{align}
    In particular note the latter convergences imply that $t_k^*t_k \geq f \geq s_nes_n^*$ for large $n, k \in \mathbb{N}$, and whence $(s_nes_n^*)^\uparrow \in D_{t_k^*t_k}^\theta$. Therefore
    \begin{equation}
      e \drel s_n e s_n^* \drel t_k s_n e s_n^* t_k^* \;\;\;\; \text{and} \;\;\;\; e \geq t_k f t_k^* \geq t_k s_n e s_n^* t_k^*. \nonumber
    \end{equation}
    Since the partial order $\geq$ restricts to equality in the $\mathcal{D}$-classes it follows that $e = t_k f t_k^* = t_k s_n e s_n^* t_k^*$ which, in particular, proves that $e = (t_k f) (t_k f)^*$ and $f = (t_k f)^* (t_k f)$, as desired.

    (\ref{thm:grpd:bisimple}) $\Rightarrow$ (\ref{thm:grpd:struc}). By Lemma~\ref{lemma:hclass} there is only one (up to isomorphism) non-trivial maximal subgroup of $S$, call it $H$. Then the map
    \begin{equation}
      \Phi \colon S \rightarrow \idempotents^+ \times H \times \idempotents^+ \sqcup \{0\}, \;\; s \mapsto
      \left\{
        \begin{array}{rl}
          \left(ss^*, \phi(s), s^*s\right) & \text{if} \;\; s \neq 0, \\
          0 & \text{if} \;\; s = 0,
        \end{array}
      \right. \nonumber
    \end{equation}
    where $\phi$ is as in Lemma~\ref{lemma:hclass} is an inverse semigroup isomorphism.

    (\ref{thm:grpd:struc}) $\Rightarrow$ (\ref{thm:grpd:minimal}). First, note that $\spectrum$ is discrete, as every filter in $\idempotents$ is of the form $e^\uparrow$. Indeed, let $\xi \in \spectrum$, and suppose that $e_n^\uparrow \rightarrow \xi$. Since no $e_n$ is equal to $0$, observe that $f \in \xi$ if, and only if, $f \geq e_1$ and $f \geq e_2$ for different $e_1 \neq e_2$. Following the inner structure of $S$ in condition~(\ref{thm:grpd:struc}) this is impossible, and hence a sequence $\{e_n^\uparrow\}_{n \in \mathbb{N}} \subset \spectrum$ is Cauchy if, and only if, it is eventually constant. Secondly, observe that $\univgrpd$ is then minimal, since for any pair of idempotents $e, f \in \idempotents^+$ we have that $e = (f, 1, e)^*(f, 1, e)$ and $f = (f, 1, e) (f, 1, e)^*$.

    The fact that $C_r^*(S)$ is stably finite follows from the fact it is the inductive limit of unital C*-algebras with faithful tracial states, and the argument will hence be completed via Proposition~\ref{prop:sf}. By the structure given in~(\ref{thm:grpd:struc}) it is not hard to see that $S$ is finitely generated (modulo $H$) if, and only if, $\idempotents^+$ is a finite set, which is equivalent to $C_r^*(S)$ being unital. Moreover, in such case there are plenty of faithful tracial states. For instance, the functional
    \begin{equation}
      \tau \colon C_r^*(S) \rightarrow \mathbb{C}, \quad \sum_{s \in S} a_s v_s \mapsto \frac{1}{2} a_0 + \left(\frac{1}{2 \left|\idempotents^+\right|} + \frac{1}{2}\right) \sum_{e \in \idempotents^+} a_e \nonumber
    \end{equation}
    can be proven to be a faithful tracial state. In general, when $\idempotents^+$ is infinite we have that $C_r^*(S)$ is no longer unital, but it is the inductive limit of the sub-C*-algebras $C^*(\langle F_n\rangle) \subset C^*_r(S)$, where $F_n \subset S$ are an increasing sequence of finite subsets of $S$.
  \end{proof}

  An immediate consequence of Theorem~\ref{thm:grpdmin} is the opposite implication of Proposition~\ref{prop:finite}.
  \begin{corollary}
    Let $S$ be an inverse semigroup such that $\univgrpd$ is minimal. Then $C_r^*(S)$ is finite if, and only if, the partial order of $S$ restricts to equality in the $\mathcal{D}$-classes.
  \end{corollary}

  Theorem~\ref{thm:grpdmin} also allows to show which inverse semigroups with minimal groupoids have quasi-diagonal reduced C*-algebras. Moreover, it shows the class of inverse semigroups whose universal groupoids meet the hypothesis of~\cite[Theorem~6.5]{RS19} is rather limited (compare the following with Example~\ref{ex:bisimple} above).
  \begin{theorem} \label{thm:minimalqd}
    Let $S$ be an inverse semigroup with minimal universal groupoid $\univgrpd$. The following assertions are equivalent:
    \begin{enumerate}
      \item \label{thm:minimalqd:qd} $C_r^*(S)$ is quasi-diagonal.
      \item \label{thm:minimalqd:subgrp} $C_r^*(S)$ is finite and every subgroup of $S$ is amenable.
      \item \label{thm:minimalqd:struc} Either $S$ is an amenable group or $S$ is isomorphic to $\idempotents^+ \times H \times \idempotents^+ \sqcup \{0\}$, where $H$ is an amenable group, $\idempotents^+$ is a non-empty set and
        \begin{equation}
          (f_2, h_2, e_2)(f_1, h_1, e_1) := 
            \left\{
              \begin{array}{rl}
                (f_2, h_2 h_1, e_1) & \text{if} \;\; e_2 = f_1, \\
                0 & \text{otherwise}.
              \end{array}
            \right. \nonumber
        \end{equation}
    \end{enumerate}
  \end{theorem}
  \begin{proof}
    In order to show (\ref{thm:minimalqd:qd}) $\Rightarrow$ (\ref{thm:minimalqd:subgrp}) it suffices to prove that $H$ is amenable, where $H$ is as in Theorem~\ref{thm:grpdmin}~(\ref{thm:grpd:struc}). This, in turn, follows from Theorem~\ref{thm:subgrp}. Indeed, note that $H \subset S$ is isolated, since the spectrum $\spectrum$ is discrete (this follows from the discussion of (\ref{thm:grpd:struc}) $\Rightarrow$ (\ref{thm:grpd:minimal}) in Theorem~\ref{thm:grpdmin}). It is also well known that amenability passes to subgroups.

    The implication (\ref{thm:minimalqd:subgrp}) $\Rightarrow$ (\ref{thm:minimalqd:struc}) is a direct consequence of Theorem~\ref{thm:grpdmin}.

    In order to prove (\ref{thm:minimalqd:struc}) $\Rightarrow$ (\ref{thm:minimalqd:qd}), observe that if $S$ is a group then the claim follows from~\cite[Corollary~C]{TWW17}. In case $S$ is not a group let $\idempotents^+$ and $H$ be as in the hypothesis, and let $e_0 \in H \cap \idempotents$ be the unit of $H$. Since $H$ is amenable, again by~\cite[Corollary~C]{TWW17}, let $p_n$ be a sequence of finite rank projections converging strongly to the identity of $\ell^2(H)$ and such that $||\lambda_h p_n - p_n \lambda_h|| \rightarrow 0$ for every $h \in H$. Fix an increasing sequence of finite subsets $F_n \subset \idempotents^+$, and consider the projections
    \begin{equation}
      q_n := v_0 + \sum_{e, f \in F_n} w_{(e, 1, e_0)} v_{(f, 1, e_0)} p_n v_{(f, 1, e_0)}^* w_{(e, 1, e_0)}^*. \nonumber
    \end{equation}
    It is then not hard to see that $q_n$ converges strongly to the identity of $\ell^2(S)$ and is asymptotically central in $C_r^*(S)$ (see Lemma~\ref{lemma:qd-dclass} below for a similar argument).
  \end{proof}

  The class of inverse semigroups covered by Theorem~\ref{thm:minimalqd} is interesting (and already appeared in Example~\ref{ex:bisimple}), as it allows us to construct inverse semigroups that have quasi-diagonal C*-algebras with no faithful traces. We end the section with such an example.
  \begin{example} \label{ex:qdnotr}
    Let $\idempotents^+$ be any countably infinite set, and consider the set $T := \idempotents^+ \times \idempotents^+ \sqcup \{0\}$ equipped with the operation $(f_2, e_2) (f_1, e_1) = (f_2, e_1)$ whenever $e_2 = f_1$, and $0$ otherwise. Note that, by Theorem~\ref{thm:grpdmin}, $T$ is $0$-bisimple and its partial order restricts to equality in the $\mathcal{D}$-classes. As the only subgroup of $T$ is the trivial group, by Theorem~\ref{thm:minimalqd} the C*-algebra $C_r^*(T)$ is quasi-diagonal. Consider now $S := T \sqcup \{1\}$, where $1$ acts as a unit in $T$. It is routine to show that then $C_r^*(S) = C_r^*(T) \oplus \mathbb{C}$, and hence is also a quasi-diagonal C*-algebra. However, observe that $C_r^*(S)$ has no faithful tracial state. Indeed, any tracial state $\tau \colon C_r^*(S) \rightarrow \mathbb{C}$ satisfies that $\tau(v_{(f, e)}) = \tau(v_{(f', e')}) = \delta \geq 0$ for every pair $(f, e), (f', e') \in S$. Putting $\rho := \tau(v_0) < \delta$, let $k \in \mathbb{N}$ be such that $k (\delta - \rho) > 1$, and note that
    \begin{equation}
      1 = \tau(1) \geq \tau\left(v_0\right) + \sum_{i = 1}^k \tau\left(v_{(f_i, e_i)} - v_0\right) = \rho + k \left(\delta - \rho\right) > 1. \nonumber
    \end{equation}
    Thus, any tracial state in $C_r^*(S)$ must vanish in $C_r^*(T)$, and hence is non-faithful.
  \end{example}

  \section{\texorpdfstring{$\mathcal{D}$}{D}-classes and quasi-diagonality} \label{sec:dclasses}
  We now proceed onto the proof of the final contribution of the text, namely Theorem~\ref{thm:dclasses}, which constructs explicit quasi-diagonalizing projections of the reduced C*-algebra of an inverse semigroup satisfying certain characteristics. The following two lemmas are rather basic, but we include them here for completeness and future reference.
  \begin{lemma} \label{lemma:qd-hclassmult}
    Let $S$ be an inverse semigroup, $s \in S$ be an element and $H \subset S$ be an $\mathcal{H}$-class. Then the following hold:
    \begin{enumerate}
      \item \label{item:lemma:qd-hclassmult:1} $H \subset s^*s \cdot S$ if, and only if, $H \cap s^*s \cdot S \neq \emptyset$.
      \item \label{item:lemma:qd-hclassmult:2} If $H \subset s^*s \cdot S$ then $sH$ is an $\mathcal{H}$-class within the same $\mathcal{L}$-class as $H$.
    \end{enumerate}
  \end{lemma}
  \begin{proof}
    Assertion~(\ref{item:lemma:qd-hclassmult:1}) follows from the observation that $h_1h_1^* = h_2h_2^*$ for every $h_1, h_2 \in H$. Thus if $h_1 \in H \cap s^*s \cdot S$ then $h_2h_2^* = h_1h_1^* \leq s^*s$, proving that $h_2 \in s^*s \cdot S$ for any $h_2 \in H$. Assertion~(\ref{item:lemma:qd-hclassmult:2}) can be proved similarly.
  \end{proof}
  \begin{lemma} \label{lemma:qd-normequal}
    Let $a, v \in \mathcal{B}(\mathcal{H})$. Suppose that $v$ is a partial isometry and $a(1 - vv^*) = 0$. Then $||a|| = ||av||$.
  \end{lemma}
  \begin{proof}
    Note that $a = avv^* + a(1 - vv^*) = avv^*$, and therefore $||a|| = ||avv^*|| \leq ||av||$.
  \end{proof}

  The upcoming lemma is the main technical result of the section, and the main tool to prove Theorem~\ref{thm:dclasses} below.
  \begin{lemma} \label{lemma:qd-dclass}
    Let $S$ be a discrete inverse semigroup, and let $D \subset S$ be a $\mathcal{D}$-class with finitely many projections. Fix an idempotent $e_0 \in D$, and suppose its $\mathcal{H}$-class is amenable (as a group). Then there is a sequence of finite dimensional orthogonal projections $\{q_n\}_{n \in \mathbb{N}} \subset \mathcal{B}(\ell^2(S))$ satisfying the following conditions:
    \begin{enumerate}
      \item \label{item-lemma:qd-dclass:below} $q_n \cdot 1_D = q_n = 1_D \cdot q_n$, where $1_D$ denotes the characteristic function of $D$.
      \item \label{item-lemma:qd-dclass:sot} $q_n$ converges strongly to $1_D$.
      \item \label{item-lemma:qd-dclass:comm} $||v_s q_n - q_n v_s|| \rightarrow 0$ as $n \rightarrow \infty$ for every $v_s \in C_r^*(S)$.
    \end{enumerate}
  \end{lemma}
  \begin{proof}
    Let $\idempotents_D = \idempotents \cap D$ be the set of idempotents within $D$, which is finite by assumption. For each $e \in \idempotents_D$ denote by $L_e$ and $R^e$ respectively the $\mathcal{L}$-class and $\mathcal{R}$-class of $e$. Likewise, let $H_e^f = L_e \cap R^f$ be the different $\mathcal{H}$-classes in $D$, and in particular note that $e \in H_e^e$. As in the proof of Lemma~\ref{lemma:hclass}, let $r_{e_0} := e_0$, and for each other $f \in \idempotents_D$ fix an element $r_f \in H_{e_0}^f$. Since $H_{e_0}^{e_0}$ is an amenable group, by \cite[Corollary~C]{TWW17} there is a sequence of projections $\{p_n\}_{n \in \mathbb{N}}$ witnessing the quasi-diagonality of $C_r^*(H_{e_0}^{e_0})$. For each pair $e, f \in \idempotents_D$ let $p_n^{e,f} \in \mathcal{B}(\ell^2(S))$ be given by
    \begin{equation}
      p_n^{e,f} = w_{r_e} v_{r_f} p_n v_{r_f}^* w_{r_e}^*, \nonumber
    \end{equation}
    where $v, w$ are, respectively, the left and right regular representations of $S$. Note that $p_n = p_n^{e_0,e_0}$, following from the choice of $r_{e_0} = e_0$. It is then clear that $p_n^{e,f}$ are finite dimensional orthogonal projections whose range sits within $\ell^2 (H_e^f)$, and hence $p_n^{e,f} p_n^{e',f'} = 0$ whenever $(e, f) \neq (e',f')$. In addition, standard computations show that the following commutation relations hold:
    \begin{equation}
      v_{r_f} p_n^{e_0,e_0} = p_n^{e_0,f} v_{r_f} \;\;\; \text{and} \;\;\; w_{r_e} p_n^{e_0,f} = p_n^{e,f} w_{r_e}. \nonumber
    \end{equation}
    In this context, let
    \begin{equation}
      q_n := \sum_{e, f \in \idempotents_D} p_n^{e,f}, \nonumber
    \end{equation}
    which is a finite dimensional orthogonal projection by the above discussion. Moreover, since $p_n$ converges strongly to the identity of $\ell^2 (H_{e_0}^{e_0})$, so does $q_n$ converge strongly to the identity of $\ell^2(D)$. Thus all we have left to do is prove that $q_n$ asymptotically commutes with every $v_s \in C_r^*(S)$. For this, due to the built-in orthogonality of $p_n^{e,f}$ and $p_n^{e',f'}$:
    \begin{equation} \label{eq-lemma:qd-dclass:norm}
      \left|\left| v_s \, q_n - q_n \, v_s\right|\right| \; = \; \left|\left| \sum_{e, f \in \idempotents_D} v_s \, p_n^{e,f} - p_n^{e,f} \, v_s \right|\right| \; = \; \max_{e \in \idempotents_D} \; \max_{\substack{f \in \idempotents_D \\ f \leq s^*s}} \; \left|\left|v_s \, p_n^{e,f} - p_n^{e,sfs^*} \, v_s\right|\right|
    \end{equation}
    where, for the last equality, we have used that $v_s p_n^{e,f} = 0$ unless $f \leq s^*s$. Seen from the point of view of the $\mathcal{H}$-classes themselves this means that $H_e^f \subset s^*s \cdot S$ which, by Lemma~\ref{lemma:qd-hclassmult}, is equivalent to $H_e^f \cap s^*s \cdot S \neq \emptyset$. By the same Lemma~\ref{lemma:qd-hclassmult} observe that, in this case, $sH_e^f$ is a (possibly different) $\mathcal{H}$-class in the same $\mathcal{L}$-class as $H_e^f$, and one can check that $s H_e^f = H_e^{sfs^*}$. By Eq.~(\ref{eq-lemma:qd-dclass:norm}) it is enough to show that $||v_s p_n^{e,f} - p_n^{e,sfs^*} v_s|| \rightarrow 0$ for every ordered pair $e, f \in \idempotents_D$ where $f \leq s^*s$. To this end, observe that
    \begin{align}
      \left|\left|v_s p_n^{e,f} - p_n^{e,sfs^*} v_s\right|\right| & = \left|\left|v_s w_{r_e} p_n^{e_0,f} w_{r_e}^* - w_{r_e} p_n^{e_0,sfs^*} w_{r_e}^* v_s\right|\right| \leq \left|\left|v_s p_n^{e_0,f} - p_n^{e_0,sfs^*} v_s\right|\right| \nonumber \\
      & = \left|\left|\left(v_s p_n^{e_0,f} - p_n^{e_0,sfs^*} v_s\right) v_{r_f} \right|\right| = \left|\left|v_{sr_f} p_n^{e_0,e_0} - p_n^{e_0,sfs^*} v_{sr_f}\right|\right|, \nonumber
    \end{align}
    where, for the prior to last equality, we have used Lemma~\ref{lemma:qd-normequal}. Indeed, if $a = v_s p_n^{e_0,f} - p_n^{e_0,sfs^*} v_s$ and $v = v_{r_f}$ then $a(1 - vv^*) = 0$ by a simple computation. Therefore, by renaming $t := sr_f$ note that $t \in H_{e_0}^{sfs^*} = H_{e_0}^{tt^*}$. Furthermore
    \begin{equation} \label{eq-lemma:qd-dclasscenter}
      v_{sr_f} p_n^{e_0,e_0} = v_{r_{tt^*}} v_{r^*_{tt^*} t} \, p_n \;\;\; \text{and} \;\;\; p_n^{e_0,sfs^*} v_{sr_f} = p_n^{e_0,tt^*} v_{r_{tt^*}} v_{r_{tt^*}^*} v_t = v_{r_{tt^*}} p_n v_{r^*_{tt^*} t}.
    \end{equation}
    Thus, the fact that $q_n$ is asymptotically central in $C_r^*(S)$ follows from Eq.~(\ref{eq-lemma:qd-dclasscenter}) and the fact that $p_n$ is asymptotically central in $C_r^*(H_{e_0}^{e_0})$:
    \begin{equation}
      \left|\left|v_s p_n^{e,f} - p_n^{e, sfs^*} v_s\right|\right| \leq \left|\left|v_{r_{tt^*}} \left(v_{r^*_{tt^*} t} p_n - p_n v_{r^*_{tt^*}t}\right) \right|\right| \leq \left|\left|v_{r^*_{tt^*} t} p_n - p_n v_{r^*_{tt^*}t} \right|\right| \xrightarrow{n \rightarrow \infty} 0, \nonumber 
    \end{equation}
    since $r^*_{tt^*} t \in H_{e_0}^{e_0}$.
  \end{proof}

  \begin{theorem} \label{thm:dclasses}
    Let $S$ be a countable and discrete inverse semigroup. Suppose the following holds:
    \begin{enumerate}
      \item \label{thm:qd-findclasses:d} No $\mathcal{D}$-class in $S$ has infinitely many projections.
      \item \label{thm:qd-findclasses:a} Every subgroup of $S$ is amenable.
    \end{enumerate}
    Then $C_r^*(S)$ is a quasi-diagonal C*-algebra.
  \end{theorem}
  \begin{proof}
    Let $\mathfrak{D}$ be the, possibly infinite, set of $\mathcal{D}$-classes of $S$. As it is countable we may write it as $\mathfrak{D} = \cup_{n \in \mathbb{N}} \mathfrak{D}_n$, where $\mathfrak{D}_n \subset \mathfrak{D}_{n + 1}$ are finite sets. Likewise, write $S = \cup_{n \in \mathbb{N}} \mathcal{F}_n$ as an increasing sequence of finite subsets. Now, for every $n \in \mathbb{N}$ and $D \in \mathfrak{D}_n$, let $\{q_n^D\}_{n \in \mathbb{N}}$ be as in the conclusion of Lemma~\ref{lemma:qd-dclass}. In addition, note we may suppose that
    \begin{equation} \label{eq-thm:qd-findclasses}
      \max_{s \in \mathcal{F}_n} \; \max_{D \in \mathfrak{D}_n} \; \left|\left|v_s \, q_n^D - q_n^D \, v_s \right|\right| \leq 1/n
    \end{equation}
    for every $n \in \mathbb{N}$. Consider
    \begin{equation}
      q_n := \sum_{D \in \mathfrak{D}_n} q_n^D. \nonumber
    \end{equation}
    It is clear that $q_n$ is a finite dimensional orthogonal projection. Moreover, since $q_n^D$ strongly converges to $1_D$ for every $\mathcal{D}$-class $D$, it follows that so does $q_n$ strongly converge to $1$. Lastly, $q_n$ is also asymptotically central by Eq.~(\ref{eq-thm:qd-findclasses}):
    \begin{equation}
      \max_{s \in \mathcal{F}_n} \; \left|\left| v_s q_n - q_n v_s\right|\right| = \max_{s \in \mathcal{F}_n} \; \max_{\substack{D \in \mathfrak{D}_n}} \; \left|\left| v_s \, q_n^D - q_n^D \, v_s\right|\right| \leq 1/n \xrightarrow{n \rightarrow \infty} 0. \nonumber
    \end{equation}
    Note that it is crucial that $q_n^D$ and $q_n^{D'}$ are orthogonal projections when $D$ and $D'$ are different $\mathcal{D}$-classes (see item~(\ref{item-lemma:qd-dclass:below}) in Lemma~\ref{lemma:qd-dclass}).
  \end{proof}

  \begin{example}
    Recall from Example~\ref{ex:groupbundle} the notion of a Clifford semigroup $S = \sqcup_{e \in E} G_e$, and note that in these semigroups the $\mathcal{D}$-classes have exactly one idempotent (namely the unit of $G_e$). By Theorem~\ref{thm:dclasses}, the C*-algebra $C_r^*(S)$ is quasi-diagonal whenever every group $G_e$ is amenable.
  \end{example}
  The following example actually shows that the converse of Theorem~\ref{thm:dclasses} does not hold in general. Moreover, it also provides an inverse semigroup with quasi-diagonal reduced C*-algebras with a non-amenable subgroup, proving that the conditions of Theorem~\ref{thm:subgrp} are sharp.
  \begin{example} \label{ex:qd-nonfl}
    Let $\mathbb{F}_2$ be the non-abelian free group on rank $2$. Fix a descending chain of finite index normal subgroups $\{N_k\}_{k \in \mathbb{N}}$ with trivial intersection, and denote the quotient maps by $q_k \colon \mathbb{F}_2 \rightarrow H_k$. For notation reasons, let $N_\infty := \{1\}$, and let $H_\infty := \mathbb{F}_2$. Consider $S := \sqcup_{k \in \mathbb{N} \cup \{\infty\}} H_k$, where
    \begin{equation}
      q_k\left(g\right) \cdot q_{k'}\left(h\right) := q_{\min\{k, k'\}}\left(gh\right). \nonumber
    \end{equation}
    It is then routine to check that $S$ is an inverse semigroup in which every idempotent is central, as these are just the identities of $H_k$. We claim that $C_r^*(S)$ is quasi-diagonal and has a non-amenable subgroup, namely $H_\infty = \mathbb{F}_2$. In order to do this we shall construct an arbitrarily large finite rank projection $p$ that almost commutes with $v_s$ for predetermined elements $s \in S$.

    Given $r, \varepsilon > 0$ let $n, m \in \mathbb{N}$ be so that $n \geq r/\varepsilon$ and such that $q_m$ is injective on the ball $B_{n+r} \subset \mathbb{F}_2$. Then for every $x \in B_n$ let
    \begin{equation}
      w_{q_m(x)} := \left(\frac{\ell\left(x\right)}{n}\right)^{1/2} \delta_x + \left(\frac{n - \ell\left(x\right)}{n}\right)^{1/2} \delta_{q_m\left(x\right)} \nonumber
    \end{equation}
    while if $q_m(x) \in H_m \setminus q_m(B_n)$ we let $w_{q_m(x)} := \delta_{q_m\left(x\right)}$. Note that by $\ell(x)$ we mean the length of $x \in \mathbb{F}_2$, i.e., the distance between $1$ and $x$ in a fixed Cayley graph of $\mathbb{F}_2$. Lastly, let $p$ be the orthogonal projection onto $\text{span}\{w_{q_m(x)}\} + \ell^2(H_0 \cup \dots \cup H_{m-1})$, which is a finite dimensional space since $H_k$ are all finite groups when $k < \infty$. Observe that $\{w_{q_m(x)}\}$ forms a family of pairwise orthogonal unit vectors in $\ell^2(S)$, and therefore
    \begin{align}
      \left|\left| v_s p - p v_s \right|\right| = \max\left\{\left|\left| v_s w_{q_m(x)} - w_{q_m(sx)} \right|\right| \;\; \text{where} \;\; q_m(x) \in H_m\right\} \leq r/n \leq \varepsilon, \nonumber
    \end{align}
    which proves the claim.

    Note that, by the main result of~\cite{Willett15}, $C_r^*(S)$ may actually be $C_{full}^*(\mathbb{F}_2)$ when $G = \mathbb{F}_2$ and the sequence $\{N_k\}_{k \in \mathbb{N}}$ is appropriately chosen. Therefore this construction recovers and generalizes the well-known fact that $C_{full}^*(\mathbb{F}_2)$ is quasi-diagonal.
  \end{example}

  We end the paper with two remarks.
  \begin{remark}
    Observe that the very same proof of Theorem~\ref{thm:dclasses} goes through even when the $\mathcal{D}$-classes are not finite but only \textit{sparse}, that is, every $\mathcal{D}$-class in a finitely generated sub-semigroup has only finitely many idempotents. Indeed, since quasi-diagonality is a local notion, i.e., depends only on finitely many elements at a time, we could work with sparse $\mathcal{D}$-classes.
  \end{remark}
  \begin{remark}
    Observe, furthermore, that both conditions~(\ref{thm:qd-findclasses:d}) and~(\ref{thm:qd-findclasses:a}) in Theorem~\ref{thm:dclasses} can be seen geometrically in the semigroup $S$. Indeed, following~\cite{CMS22}, one can equip $S$ with a proper and right subinvariant metric, and then condition~(\ref{thm:qd-findclasses:d}) means the Sch\"{u}tzenberger graphs of $S$ are finite unions of group Cayley graphs. They are, thus, coarsely equivalent to their own unique $\mathcal{H}$-class. Moreover, that $\mathcal{H}$-class is amenable by condition~(\ref{thm:qd-findclasses:a}).
  \end{remark}

  \providecommand{\bysame}{\leavevmode\hbox to3em{\hrulefill}\thinspace}
  
\end{document}